\documentclass{amsart}
\usepackage[mathscr]{eucal}
\usepackage{amsmath,amssymb,hyperref}
\usepackage{amstext}
\usepackage{graphics}
\usepackage{xcolor}
\usepackage{mathrsfs}
\usepackage{epstopdf}

\usepackage{tikz}
\usetikzlibrary{matrix}

\usepackage[all]{xy}

\usepackage[T1]{fontenc}
\usepackage[utf8]{inputenc}

\newtheorem{thm}{Theorem}[section]

\newtheorem{cor}[thm]{Corollary}
\newtheorem{prop}[thm]{Proposition}

\newtheorem{lemma}[thm]{Lema}

\theoremstyle{definition}

\newtheorem{definition}[thm]{Definition}
\newtheorem{remark}[thm]{Remark}
\newtheorem{example}[thm]{Example}

\newtheorem*{teo}{Theorem}
\newtheorem{Teorema}{Theorem}

\DeclareMathOperator{\Hom}{Hom}

\def\X0{X^{\circ}}

\def\Y0{Y^{\circ}}

\input xy
\xyoption{all}

\addtocounter{section}{0}             % Start with section 1
\numberwithin{equation}{section}       % Number formulas within sections

\begin{document}

\title[Invariant curves for holomorphic foliations on singular surfaces]
{Invariant curves for holomorphic foliations on singular surfaces}

%\author[J.V. Pereira]{Jorge Vit\'{o}rio PEREIRA}
%\address{IMPA, Estrada Dona Castorina, 110, Horto, Rio de Janeiro,
%Brasil}
%\email{jvp@impa.br}

\author[E. A. Santos]{Edileno de Almeida SANTOS}
\address{Faculdade de Ciências Exatas e Tecnologia, Universidade Federal da Grande Dourados - UFGD, Rodovia Dourados - Itahum, Km 12 - Cidade Universitátia, Dourados - MS,
Brazil}
\email{edilenosantos@ufgd.edu.br}

\subjclass{37F75} \keywords{Foliations, Invariant Curves, Birational Geometry}

%\thanks{The author would like to thanks J. V. Pereira for incentive and valuable conversations.}

\begin{abstract}
The Separatrix Theorem of C. Camacho and P. Sad guarantees the existence of invariant curve (separatrix) passing through the singularity of germ of holomorphic foliation on complex surface, when the surface underlying the foliation is smooth or when it is singular and the dual graph of resolution surface singularity is a tree. Under some assumptions, we obtain existence of separatrix even when the resolution dual graph of the surface singular point is not a tree. It will be necessary to require an extra condition of the foliation, namely, absence of saddle-node in its reduction of singularities.
\end{abstract}

\maketitle

\setcounter{tocdepth}{1}
\sloppy
%\tableofcontents

The author would like to thanks J. V. Pereira for incentive, valuable ideas and conversations.

\section{Introduction}\label{S:Introduction}
In the local theory of holomorphic foliations, one of the first questions was about the existence of invariant curve ({\it separatrix}) through isolated singularity of a holomorphic vector field on $(\mathbb C^2,0)$. This problem was posed and investigated by C. Briot and J. Bouquet in 1854. The definitive solution was given many years later, in 1982, when C. Camacho and P. Sad showed in \cite{CSS} the following

\begin{teo}[Separatrix Theorem] {\it Let $v$ be a holomorphic vector field defined on neighbourhood of $0\in \mathbb C^2$ and with isolated singularity at the origin. Then there exists an invariant curve for $v$ through $0\in \mathbb C^2$.}
\end{teo}

In 1988, C. Camacho showed in \cite{C} an extension of this result to include foliations on some singular surfaces:

\begin{teo} {\it Let $X$ be a complex normal irreducible surface. Suppose the dual graph of resolution singularity at $p\in X$ is a tree. Then any germ of holomorphic foliation, singular at $p$, has an invariant curve through $p$.}
\end{teo}

A very elegant proof was given by M. Sebastiani in \cite{Sebastiani} (see also the exposition of M. Brunella in \cite{Brunella}). In the non singular case, M. Toma in \cite{Toma} and J. Cano in \cite{Cano} gives alternative proofs. Here we explore the surface singular case, but without the hypothesis on the dual graph. %In this way we find some extensions to include singularities whose dual graph of resolution is not a tree.

A {\it holomorphic foliation} $\mathcal F$ on a smooth complex surface $X$ is given by an open covering $\{ U_i \} $ and holomorphic vector fields with isolated singularities $X_i$ over each $U_i$ such that whenever $U_i\cap U_j\neq \emptyset$ there exists an invertible holomorphic function $g_{ij}$ satisfying $X_i=g_{ij}X_j$. The collection $\{ g_{ij}^{-1} \}$ defines the holomorphic line bundle  $T\mathcal F$, called the {\it tangent bundle} of  $\mathcal F$. The dual os  $T\mathcal F$ is the {\it cotangente bundle}  $T^*\mathcal F$, also called the {\it canonical bundle}  $K\mathcal F$. Recall that a {\it reduced foliation}  $\mathcal F$ is a foliation such that every singularity $p$ is reduced in Seidenberg's sense: there is a vector field $X$ generating  $\mathcal F$ on neighbourhood of $p$ such that the eigenvalues of the linear part of $X$ are $1$ and $\lambda$, where $\lambda$ is not a positive rational number. If $\lambda=0$, the singularity $p$ is called a {\it sadlle-node}; otherwise it is called a {\it non degenerate singularity}. The well-known {\it Seidenberg's Theorem} states that after a finite sequence of blowing-ups over a singularity $p$, the induced foliation is reduced  along the strict transform of $p$ (see {\bf Theorem 1}, \cite{Brunella}, page 13). This process is called {\it reduction of singularities}.

If $X$ is a surface with only normal singularities, a {\it holomorphic foliation} on $X$ is a holomorphic foliation on $X-Sing(X)$. The {\it tangent} sheaf $T_{\mathcal F}$ and the {\it normal} sheaf $N_{\mathcal F}$ are defined by taking the direct images via the inclusion $X-Sing(X)\rightarrow X$ of the tangent and normal sheaves of the underlying foliation on $X-Sing(X)$.

In \cite{C}, C. Camacho gives an example of non existence of separatrix. Let $\pi:Y\rightarrow X$ be a resolution of the normal surface singularity $p\in X$, where $\pi^{-1}(p)=E=E_1+E_2+E_3$ is a curve with simple normal crossings and each $E_j$ is a smooth rational curve, with self-intersections $E_1^2=-2$, $E_2^2=-2$ and $E_3^2=-3$. To find a reduced foliation without saddle-node on neighbourhood of the curve $E$ such that $E$ is $\mathcal F$-invariant and $Sing(\mathcal F)\cap E=Sing(E)$, consider the dual graph $\Gamma$ of $E$, which is a cycle, as illustrated bellow:

%\begin{figure}[!htb]
%\centering
%\includegraphics[scale=0.3]{ciclo_tablet.png}
%\caption{Cycle of three invariant curves.}
%\end{figure}

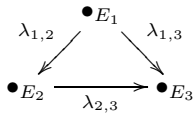
$$\xymatrix@R5pt@C7pt{
 & \ar[ldd]_{\lambda_{1,2}} \bullet_{E_1} \ar[rdd]^{\lambda_{1,3}}  &  \\ 
 &  & \\ 
\bullet_{E_2} \ar[rr]_{\lambda_{2,3}}&  & \bullet_{E_3}
} $$
where $\lambda_{ij}=CS(\mathcal{F}, C_i,p_{ij})$ is the {\it Camacho-Sad Index} (see \cite{Brunella}), $\{p_{ij}\}=E_i\cap E_j$. Thus $\lambda_{1,2}+\lambda_{1,3}=E_1^2$, $\frac{1}{\lambda_{1,2}}+\lambda_{2,3}=E_2^2
$, $\frac{1}{\lambda_{2,3}}+\frac{1}{\lambda_{1,3}}=E_3^2.$
These conditions are enough to construct the foliation on neighbourhood of the curve $E$ as desired (if the $\lambda_{ij}$ are not rational numbers, then the foliation will be reduced and without saddle-nod; see proof by Lins-Neto of the main theorem in \cite{Lins-Neto}). In this way, if we contract the exceptional curve $E$ to the surface singular point $p$, then we have a foliation without separatrix trough the singularity $p$.

Hence we can ask now:

{\it What conditions can we require from the foliation to exist separatrix even when the resolution dual graph is not a tree?}

%The sections \ref{chapter: R} and \ref{chapter: ES} develop answers to the above question.
First, in Section \ref{chapter: R} we investigate relationships between foliations and their invariant compact curves by means of the logarithmic co-normal bundle, in particular we analyse the situation in which such a bundle restricted to the curve is trivial, which leads us to the notions of {\it residues} and {\it residual representation}. %(Theorems \ref{T:Logarítmica Global-Local} and \ref{T:Logarítmica Local-Global}).
From these ideas, in Section \ref{chapter: ES} we explore the Separatrix Theorem. The main result is:

\begin{Teorema}[Theorem \ref{T:Gorenstein}]\label{T:Teorema 1}
{\it Let $\mathcal F$ be a foliation on the normal singular surface $X$. If the foliation has no saddle-node in its resolution over the singularity $p\in X$ and the normal sheaf $N_{\mathcal F}$ is $\mathbb Q$-Gorenstein, then $\mathcal F$ has a separatrix through $p$.}
%{\it Let $\mathcal F$ be a foliation on the normal singular surface $X$. Let $f:Y\rightarrow X$ be the resolution of the singularity $p\in X$ such that the induced foliation $\mathcal G=f^*\mathcal F$ has only reduced singularities in $E=f^{-1}(p)$. If $\mathcal G$ has no saddle-node in $E$ and the normal sheaf of $\mathcal F$ is $\mathbb Q$-Gorenstein, then $\mathcal F$ has separatrix trough $p$.}
\end{Teorema}

In the above theorem, a sheaf $\mathcal S$ on a singular surface is called {\it $\mathbb Q$-Gorenstein} if there is a positive integer $k>0$ such that the $k^{th}$ tensor power $\mathcal S^{\otimes k}$ is locally trivial. We say that a foliation {\it has no saddle-node} if the induced foliation after reduction of singularities has no saddle-node.% (even at the singularities). 

In this paper we assume the reader is familiar with the theory of singular holomorphic foliations on surfaces like presented by the first chapters of \cite{Brunella}.

\section{Residues and Representation}\label{chapter: R}

\subsection{Basic Conceptions}\label{S:Conceitos Básicos}

Let $X$ be a smooth complex surface and $C\subset X$ be a curve.
\begin{definition}\label{D:Forma Logarítmica}
A {\it logaritimic $1$-form} $\omega$ on an open set $U\subset X$ with poles on $C$ is a meromorphic $1$-form on $U$ with the following property: for any $p\in U$ there is a neighborhood $V\subset U$ of $p$ such that
\[
\omega\mid_V=\omega_0+\sum_{i=1}^ng_i\frac{df_i}{f_i}
\]
where $\omega_0$ is a holormophic $1$-form on $V$, $f_i$ and $g_i$ are holomorphic functions on $V$ and each $f_i$ is a reduced equation of an irreducible component of $C\cap V$.
\end{definition}

Generically we are interesting in the situation where $C=\sum_{i=1}^nC_i\subset X$ is a {\it simple normal crossing divisor}, in which case we have the following exact sequence:
\[
0\rightarrow \Omega^1_X\rightarrow \Omega^1_X(\log C)\rightarrow \bigoplus_{i=1}^n \mathcal O_{C_i}\rightarrow 0
\]
where $\Omega^1_X(\log C)$ is the locally free sheaf (of rank $2$) of logarithmic $1$-forms (see {\bf Lemma 8.16} of \cite{Voisin}) and the last morphism is given by the {\it residue}, defined as follow: if locally $\omega\mid_V=\omega_0+\sum_{i=1}^ng_i\frac{df_i}{f_i}$,
then the residue of $\omega$ along $C_i$ is given by $Res(\omega)\vert_{C_i}=g_i\vert_{\{f_i=0\}}$ (see \cite{Brunella}, pages 78 to 83).
We can see easily that this definition is independent of our choices ($f_i$, $g_i$, $\omega_0$ etc.).

In the context of foliations we have the following definition.

\begin{definition}\label{D:Folheação Logarítmica} Let $\mathcal F$ be a foliation on $X$ and suppose that the curve $C$ is $\mathcal F$-invariant. We say that the foliation $\mathcal F$ is {\it logarithmic} on $(X,C)$ when given any $p\in C$,  $\omega$ a $1$-form defining $\mathcal F$ on neighbourhood $V$ of $p$ and $f$ a reduced equation of $C\cap V$, the meromorphic $1$-form $\frac{\omega}{f}$ is logarithmic with poles on $C$.
\end{definition}

If the curve $C$ has only  normal crossing singularities, then any reduced foliation $\mathcal F$ on $X$ tangent to $C$ is logarithmic on $(X,C)$. In that case, the bundle $$L=N^*_{\mathcal F}\otimes \mathcal O_X(C)$$
is called {\it logarithmic conormal bundle} along $C$. If the curve is a divisor with {\it simple} normal crossing singularities, then the morphism of residue induces an exact sequence:
\[
0\rightarrow N^*_{\mathcal F}\rightarrow N^*_{\mathcal F}\otimes \mathcal O_X(C)\rightarrow \bigoplus_{i=1}^n \mathcal O_{C_i}
\]
Observe that a collection of local sections $\eta_j$ of $N^*_{\mathcal F}\otimes \mathcal O_X(C)$ defined on open sets $V_j$ determines a collection of residues $Res_{V_j \cap C_i}(\eta_j)$ as local sections of $N^*_{\mathcal F}\otimes \mathcal O_X(C)\vert_{C_i}$, $i=1,...,n$.

%\begin{example}
%Let $$\omega=\lambda_1\frac{dy}{y}-\lambda_2\frac{dx}{x}$$ be a {\it closed} logarithmic $1$-form. Then
%$$Res_{\{x=0\}}(\omega)=-\lambda_2$$
%and
%$$Res_{\{y=0\}}(\omega)=\lambda_1.$$
%\end{example}

The following normal form enable us to calculate the quotient of residues for a reduced non degenerate singularity:

\begin{thm}[{\bf THÉORÈME 1}, \cite{MM}, page 521]\label{M-M}
Let $\mathcal F$ be a foliation given by a holomorphic $1$-form $\omega$  on a neighbourhood of $0\in \mathbb C^2$ with isolated singularity at $0$ and linear part $\lambda_1xdy-\lambda_2ydx$. Suppose that $\lambda_1, \lambda_2$ are not zero and $\lambda_1/\lambda_2$ and $\lambda_2/\lambda_1$ are not integers $>1$. Then there is a change of coordinates in a neighbourhood of the origin such that, in the new coordinates, the foliation is given by the $1$-form
\[
\lambda_1x(1+xy(...))dy-\lambda_2y(1+xy(...))dx
\] 
\end{thm}

\begin{example}\label{Exemplo Resíduo}
As a consequence of the above theorem, we conclude that if the foliation $\mathcal F$ has a reduced non degenerate singularity at $0\in \mathbb C^2$, then we can find a logarithmic expression for the foliation:
\[
\eta=\lambda_1(1+xy(...))\frac{dy}{y}-\lambda_2(1+xy(...))\frac {dx}{x}
\]
in such a way that $Res_{\{x=0\}}(\eta)=-\lambda_2$ and $Res_{\{y=0\}}(\eta)=\lambda_1.$
Hence the quotient of residues is minus the quotient of eigenvalues. 
\end{example}

Note that if the curve $C$ is compact and $\omega$ is a section of $N^*_{\mathcal F}\otimes \mathcal O_X(C)$ on neighbourhood of $C$, then the residue along each irreducible component $C_i$ is constant.

Fix a compact curve $C\subset X$ with simple normal crossing singularities. Let $\mathcal F$ be a reduced foliation tangent to $C$.

\begin{prop}\label{P:Cociclo}
If the reduced foliation $\mathcal F$ is without saddle-node in $C$, then there is a neighbourhood $U$ of $C$ and logarithmic $1$-forms on open subsets of $U$ defining $N^*_{\mathcal F}\otimes \mathcal O_X(C)\vert_U$ whose residues are locally constant.
\end{prop}
\begin{proof}
Away from the singularities $Sing(\mathcal F)\cap C$, we can choose local sections of the logarithmic conormal bundle with any fixed residue in $\mathbb C^*$. In  neighbourhoods of the singular points, local sections with constant residues are obtained via Example \ref{Exemplo Resíduo}.
\end{proof}

Remember the {\it intersection formula for normal bundle} (the index $Z$ below, seen as a divisor, is defined in \cite{Brunella}, pages 24 and 25):

\begin{thm}[{\bf Proposition 3}, \cite{Brunella}, page 25]\label{T:Normal Intersecta}
Let $\mathcal F$ be a foliation on a complex surface $X$ and let $C\subset X$ be a smooth compact $\mathcal F$-invariant curve. Then we have $N_{\mathcal F}\vert_C=\mathcal N_C\otimes \mathcal O_C(Z(\mathcal F, C))
$. If $D\subset X$ is any invariant connected curve (not necessarily smooth), then $N_{\mathcal F}\cdot D= D^2+Z(\mathcal F, D)$.
\end{thm}

As a consequence of Theorem \ref{T:Normal Intersecta}, we have the following proposition:

\begin{prop}\label{P:Trivialidade}
If the reduced foliation $\mathcal F$ is without saddle-node in the invariant compact curve $C=\sum_{i=1}^n C_i$, where each $C_i$ is a smooth curve, and $Sing(\mathcal F)\cap C=Sing(C)=\bigcup_{i\neq j} C_i \cap C_j$, then $N^*_{\mathcal F}\otimes \mathcal O_X(C)\vert_{C_i}=\mathcal O_{C_i}$, $i=1,...,n$.
\end{prop}
\begin{proof}
The index $Z$ is zero at any point $p\in C_i-C\cap Sing(C)$ and equal to $1$ at the points in $C_i\cap Sing(C)$. Hence
\begin{eqnarray*}
N_{\mathcal F}\vert_{C_i}&=&\mathcal N_{C_i}\otimes\mathcal O_{C_i}(Z(\mathcal F, C_i))\nonumber\\
&=&\mathcal N_{C_i}\otimes\mathcal O_{C_i}(C_i\cap Sing(C))\nonumber\\
&=&\mathcal N_{C_i}\otimes\mathcal O_{C_i}(\cup_{j\neq i} C_j\cap C_i)\nonumber\\
&=&\mathcal O_X(C)\vert_{C_i}
\end{eqnarray*}
\end{proof}

\begin{remark} The above Proposition is also a consequence of (the proof of) Proposition \ref{P:Cociclo}. In fact, we note that the residues have co-cycle of transition equal to the co-cycle of $N^*_{\mathcal F}\otimes \mathcal O_X(C)\vert_{C_i}$. % (in neighbourhood of a singular point of $C$ we have two components for the co-cycle, each one with value given by the transitions of the correspondents residues). In particular, in each component $C_i$, the co-cycle of transition of the residues gives a co-cycle for $N^*_{\mathcal F}\otimes \mathcal O_X(C)\vert_{C_i}$.
We can choose local sections defining $N^*_{\mathcal F}\otimes \mathcal O_X(C)$ along a neighbourhood of $C_i$ with the same fixed residue in $\mathbb C^*$ (along $C_i$), thus we obtain the triviality of $N^*_{\mathcal F}\otimes \mathcal O_X(C)\vert_{C_i}$.  
\end{remark}

From our discussion above we also obtain

\begin{prop}\label{P:Constante Trivial}
Suppose that the reduced foliation $\mathcal F$ is without saddle-node at the invariant compact curve $C=\sum_{i=1}^n C_i$, where each $C_i$ is a smooth curve, and $Sing(\mathcal F)\cap C=Sing(C)=\bigcup_{i\neq j} C_i \cap C_j$. Then $N^*_{\mathcal F}\otimes \mathcal O_X(C)\vert_{C}=\mathcal O_{C}$ if, and only if, there are $\delta_1,...,\delta_n\in \mathbb C^*$, open sets $V_l$ whose union is a neighbourhood of $C$ in $X$ and sections $\eta_l$ of $N^*_{\mathcal F}\otimes \mathcal O_X(C)\vert_{V_l}$ such that
\[
V_l	\cap C_i\neq \emptyset\Rightarrow Res_{C_i}(\eta_l)=\delta_i
\]
$i=1,..., n$.
\end{prop}
\begin{proof}
Suppose that $N^*_{\mathcal F}\otimes \mathcal O_X(C)\vert_{C}=\mathcal O_{C}$. Taking the product of suitable constants by the local sections $\eta_l$ given by 	Proposition \ref{P:Cociclo}, we can assume that the co-cycle of transition of residues is constant equal to $1$. Hence, in each component $C_i$, we have
$$
V_l	\cap C_i\neq \emptyset\Rightarrow Res_{C_i}(\eta_l)=\delta_i
$$
for constants $\delta_i\in \mathbb C^*$, $i=1,...,n$.

Reciprocally, suppose that there exist local sections $\eta_l$ as above. Then clearly the co-cycle of transition of the residues is constant equal to $1$ and hence $N^*_{\mathcal F}\otimes \mathcal O_X(C)\vert_{C}$ is trivial.  
\end{proof}

%\begin{figure}[!htb]
%\centering
%\includegraphics[scale=0.3]{res.jpg}
%\caption{Quocient of residues at $p$.}
%\end{figure}

\begin{remark}\label{O:Quociente de Resíduos} Let $p\in C_1\cap C_2$, where $C_1$ and $C_2$ are smooth curves. If $C=C_1 \cup C_2$ is invariant by a reduced foliation $\mathcal F$ with reduced non degenerate singularity at $p$ and $\eta$ is a local section on neighbourhood of $p$ defining $N^*_{\mathcal F}\otimes \mathcal O_X(C)$, then the quotient
$$Res_{C_1}\eta / Res_{C_2}\eta$$
is independent of $\eta$ and equal to minus the quotient of eigenvalues at $p$. This is a consequence of Theorem \ref{M-M} (see Example \ref{Exemplo Resíduo}).
\end{remark}

\subsection{Residual Representation}\label{S:Representação Residual}

Let $\mathcal F$ be a holomorphic foliation on the smooth complex surface $X$. Suppose that $\mathcal F$ leaves invariant a curve $C=\sum_{i=1}^n C_i$ with the following properties:
\begin{enumerate}
\item $C_1,..., C_n$ are smooth curves;
\item  every singularity of $\mathcal F$ in $C$ is reduced;
\item $Sing(C)=\bigcup_{i\neq j} C_i \cap C_j$ are reduced non degenerate singularities of $\mathcal F$.
\end{enumerate}

To simplify the exposition, we assume $\#C_i\cap C_j\leq 1$ if $i\neq j$ (there is no loss of generality, because we can do this situation after some blow-ups). If $i\neq j$ and $C_i\cap C_j\neq \emptyset$, take $\{p_{ij}\}=\{p_{ji}\}=C_i\cap C_j$.

Associated to the $\mathcal F$-invariant divisor $C$ we have his dual graph $\Gamma$, where each vertex correspond to an irreducible component $C_i$ and each edge linking two vertex correspond to the intersection of them at the point $p_{ij}$. The graph is {\it directed}, with ordering given by the ordering index from the vertices, that is, the edge $p_{ij}$ linking the vertex $C_i$ to the vertex $C_j$ has the orientation from $C_i$ to $C_j$ if $i<j$. (We make an arbitrary choice of orientation, but this is innocuous.)

To each edge $p_{ij}$, with $i<j$, we associate the complex number $\delta_{ij}=-CS(\mathcal F, C_j,p_{ij})\in \mathbb C^*$. (Note that the foliation is reduced non degenerate at $p_{ij}$, hence
$$\delta_{ij}=-CS(\mathcal F, C_j,p_{ij})=-\alpha_{ij} / \beta_{ij}=Res_{C_i}/Res_{C_j}$$
where $\alpha_{ij} $, $\beta_{ij}$ are eigenvalues at $p_{ij}$ associated to a generating vector field of the foliation.)

Thus we obtain a co-homology class
$\sigma=\sigma_{(\mathcal F,C)}\in H^1(\Gamma,\mathbb C^*)$ or (equivalently) a representation $\rho=\rho_{(\mathcal F,C)}: \pi_1(\Gamma)\rightarrow \mathbb C^*$ to be called {\it residual representation}.

In some cases, the representation just presented will correspond to the obstruction for triviality of the line bundle $N^*_{\mathcal F}\otimes \mathcal O_X(C)\vert_{C}$. More precisely, we have 

\begin{prop}\label{P:Resíduo e Representação}
Suppose that the reduced foliation $\mathcal F$ is without saddle-node in the invariant compact curve $C=\sum_{i=1}^n C_i$, where each $C_i$ is a smooth curve, and $Sing(\mathcal F)\cap C=Sing(C)=\bigcup_{i\neq j} C_i \cap C_j$. Then $N^*_{\mathcal F}\otimes \mathcal O_X(C)\vert_{C}=\mathcal O_{C}$ if, and only if, the residual representation $\rho_{(\mathcal F,C)}$ is trivial.
\end{prop}
\begin{proof} If $N^*_{\mathcal F}\otimes \mathcal O_X(C)\vert_{C}=\mathcal O_{C}$, then the above Proposition (\ref{P:Constante Trivial}) plus the Remark (\ref{O:Quociente de Resíduos}) imply that the representation $\rho_{(\mathcal F, C)}$ is trivial.

Reciprocally, suppose the triviality of the residual representation. Then there exist $\delta_1,...,\delta_n\in \mathbb C^*$ such that
$$\delta_i/\delta_j=\delta_{ij}=Res_{C_i}/Res_{C_j}.$$
Thus, it is possible to choose local sections of the logarithmic co-normal bundle in neighbourhoods of the singular points of $C$ with residue $\delta_i$ through $C_i$, $i=1,...,n$ (Remark \ref{O:Quociente de Resíduos}). Away from the singularities of $C$, we can choose any residue in $\mathbb C^*$, then we take local sections with residue $\delta_i$ along $C_i$. Hence, such local sections give the trivialization of the logarithmic co-normal bundle on $C$ (by residues).
\end{proof}

\begin{example}
If the dual graph $\Gamma$ of $C$ is a {\it tree}, then the graph is {\it contractible} (that is, $\Gamma$ has the homotopy type of a point), then obviously the representation $\rho_{(\mathcal F,C)}$ is trivial.
\end{example}

\begin{definition}
A {\it cicly of smooth rational curves} (or simply a {\it cicly}) is a union of a finite number of smooth rational curves in general position (normal crossings) $C_i$, $i=1,..., m$, $m>1$, such that:
 if $m=2$, then $\# C_1 \cap C_2=2$;
 if $m>2$, then $\# C_i \cap C_{(i+1)}=\# C_1 \cap C_m=1$, $i=1,..., m-1$, otherwise $\# C_i \cap C_j=0$.
\end{definition}

\begin{example}
Let $\mathcal F$ be a foliation tangent to a cycle of smooth rational curves $C=C_1+C_2+C_3$, where $C_j^2=1$, $j=1, 2, 3$, and such that $Sing(\mathcal F)\cap C=Sing(C)$ are reduced non degenerate singularities of $\mathcal F$. (For example, consider the foliation given by $\omega=xdy-\lambda ydx$ on $\mathbb P^2$, where $\lambda \in \mathbb C-\mathbb Q_{\geq 0}$.)

%\begin{figure}[!htb]
%\centering
%\includegraphics[scale=0.3]{ciclo_tablet.png}
%\caption{Invariant cycle of three smooth rational curves.}
%\end{figure}

The dual graph of $C$ is
$$
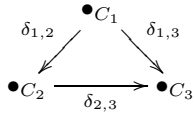
\xymatrix@R5pt@C7pt{
 & \ar[ldd]_{\delta_{1,2}} \bullet_{C_1} \ar[rdd]^{\delta_{1,3}}  &  \\ 
 &  & \\ 
\bullet_{C_2} \ar[rr]_{\delta_{2,3}}&  & \bullet_{C_3}
} $$
where the numbers $\delta_{ij}$ give the correspondent co-cycle in $H^1(\Gamma, \mathbb C^*)$ from the residual representation. Thus $\delta_{1,2}+\delta_{1,3}=\frac{1}{\delta_{1,2}}+\delta_{2,3}=\frac{1}{\delta_{2,3}}+\frac{1}{\delta_{1,3}}=-1$.
We can see easily that in this case the representation $\rho_{(\mathcal F, C)}$ is always trivial. In fact, $\delta_{1,2} \delta_{2,3}\delta_{3,1}=\delta_{1,2} (-\frac{1+\delta_{1,2}}{\delta_{1,2}})(-\frac{1}{1+\delta_{1,2}})=1$.
\end{example}

\subsection{The Picard Group of a Curve}\label{S:Grupo de Picard de Curvas}

Let $C=\sum_{i=1}^n C_i$ be a compact complex curve (connected!), where each $C_i$ is a smooth curve, and $\Gamma$ be the dual graph of $C$. We will describe the line bundles over $C$ with the property of triviality over each irreducible component $C_i$, $i=1,..., n$.

\begin{prop}\label{P:Fibrado em Curva}
Let $T(C)=\{F\in {\rm Pic(C)}; F\vert_{C_i}=\mathcal O_{C_i}, i=1,...,n\}$. Then
\[
T(C)\simeq H^1(\Gamma,\mathbb C^*)\simeq \Hom(\pi_1(\Gamma),\mathbb C^*).
\]
\end{prop}

\begin{proof}
Consider the short exact sequence
\[
1\rightarrow \mathcal O_C^*\rightarrow \bigoplus_{i=1}^n \mathcal O_{C_i}^*\rightarrow \bigoplus_{p\in Sing(C)}  \mathbb C^*\rightarrow 1
\]
where the morphism with image in $\displaystyle\bigoplus_{p\in Sing(C)}\mathbb C^*$ is given by quotient (to do this we use the index order): if $k<l$ and $q\in C_k\cap C_l$, a local section of $\bigoplus_{i=1}^n \mathcal O_{C_i}^*$ with value $\lambda_k$ at $q\in C_k$ and $\lambda_l$ at $q\in C_l$ has image with value $\lambda_k/\lambda_l$ at $q$ (as a section of $\displaystyle\bigoplus_{p\in Sing(C)}\mathbb C^*$). Hence, we have a long exact sequence in co-homology
\[
1\rightarrow \mathbb C^*\rightarrow \displaystyle\bigoplus_{i=1}^n \mathbb C^*\rightarrow \displaystyle\bigoplus_{p\in Sing(C)} \mathbb C^*\rightarrow {\rm Pic}(C)\rightarrow \displaystyle\bigoplus_{i=1}^n {\rm Pic}(C_i)\rightarrow 1.
\]
Therefore
\begin{eqnarray*}
\Hom(\pi_1(\Gamma),\mathbb C^*)&\simeq &H^1(\Gamma,\mathbb C^*)\nonumber\\
&\simeq &\frac{\bigoplus_{p\in Sing(C)} \mathbb C^*}{{\rm Im}(\bigoplus_{i=1}^n \mathbb C^*\rightarrow \bigoplus_{p\in Sing(C)} \mathbb C^*)}\nonumber\\
&\simeq &\frac{\bigoplus_{p\in Sing(C)} \mathbb C^*}{{\rm Ker}(\bigoplus_{p\in Sing(C)}\mathbb C^*\rightarrow {\rm Pic}(C))}\nonumber\\
&\simeq &{\rm Im}(\bigoplus_{p\in Sing(C)} \mathbb C^*\rightarrow {\rm Pic}(C))\nonumber\\
&\simeq &{\rm Ker}({\rm Pic}(C)\rightarrow \bigoplus_{i=1}^n {\rm Pic}(C_i))\nonumber\\
&= &T(C)
\end{eqnarray*}
\end{proof}

Thus, a line bundle $F\in {\rm Ker}({\rm Pic(C)}\rightarrow\bigoplus_{i=1}^n {\rm Pic}(C_i))={\rm Im}(\bigoplus_{p\in Sing(C)} \mathbb C^*\rightarrow {\rm Pic}(C))=T(C)$ determines an element in $H^1(\Gamma,\mathbb C^*)$ (hence a representation $\rho_F: \pi_1(\Gamma)\rightarrow \mathbb C^*$) and reciprocally an element in $H^1(\Gamma,\mathbb C^*)$ determines a line bundle in $T(C)$.

The interesting situation for us consists of a curve $C=\sum_{i=1}^n C_i$, with simple normal crossing singularities, invariant by a reduced foliation $\mathcal F$ without saddle-node and singular only at the crossing points of the curve, and the logarithmic co-normal bundle $L=N^*_{\mathcal F}\otimes \mathcal O_X(C)$ in restriction to the curve $C$, which result $F=L\vert_{C}\in {\rm Ker}( {\rm Pic}(C)\rightarrow \bigoplus_{i=1}^n {\rm Pic}(C_i))$ by Proposition \ref{P:Trivialidade}.

%At the light of what we do until now, we can write the commutative diagram bellow.

%$$\xymatrix{
%&&1\ar@{->}[r]&H^1(\Gamma, \mathbb C^*)\ar@{->}[r]&{\rm Pic}(C)\ar@{->}[r]\ar@{=}[d]&\displaystyle\bigoplus_{i=1}^n {\rm Pic}(C_i)\ar@{->}[r]\ar@{=}[d]&1\\
%1\ar@{->}[r]&\mathbb C^*\ar@{->}[r] \ar@{->}[d]&\displaystyle\bigoplus_{i=1}^n \mathbb C^*\ar@{->}[r]\ar@{->}[d]&\displaystyle\bigoplus_{p\in Sing(C)} \mathbb C^*\ar@{->}[r]\ar@{=}[d]&{\rm Pic}(C)\ar@{->}[r]&\displaystyle\bigoplus_{i=1}^n {\rm Pic}(C_i)\ar@{->}[r]&1\\
%1\ar@{->}[r]&\mathbb C_{C}^*\ar@{->}[r]\ar@{->}[d] &\displaystyle\bigoplus_{i=1}^n \mathbb C_{C_i}^*\ar@{->}[r]\ar@{->}[d]&\displaystyle\bigoplus_{p\in Sing(C)}\mathbb C^*\ar@{->}[r]\ar@{=}[d]&1\\
%1\ar@{->}[r]&\mathcal O_{C}^*\ar@{->}[r] &\displaystyle\bigoplus_{i=1}^n \mathcal O_{C_i}^*\ar@{->}[r]\ar@{->}[d]&\displaystyle\bigoplus_{p\in Sing(C)}\mathbb C^*\ar@{->}[r]&1\\
%&L\ar@{->}[r]^-{Res} &\displaystyle\bigoplus_{i=1}^n \mathcal O_{C_i}
%}$$

Finally, we can generalize the Proposition \ref{P:Resíduo e Representação}

\begin{prop}\label{P:Resíduo e Representação Generalizados}
Suppose that the reduced foliation $\mathcal F$ is without saddle-node in the invariant compact curve $C=\sum_{i=1}^n C_i$, where each $C_i$ is a smooth curve, and $ Sing(\mathcal F)\cap C=Sing(C)=\bigcup_{i\neq j} C_i \cap C_j$. Then the line bundle $F=N^*_{\mathcal F}\otimes \mathcal O_X(C)\vert_{C}$ is determined by the residual representation $\rho_{(\mathcal F,C)}$. That is, $\rho_F=\rho_{(\mathcal F, C)}$.
\end{prop}
\begin{proof}
It's enough to note that
$$F=N^*_{\mathcal F}\otimes \mathcal O_X(C)\vert_{C}\in {\rm Ker}( {\rm Pic}(C)\rightarrow \bigoplus_{i=1}^n {\rm Pic}(C_i)).$$
\end{proof}
\subsection{Residual Divisor}\label{S:Divisor Residual}

If the representation $\rho_{(\mathcal F, C)}$ is trivial, then there are $\delta_1,..., \delta_n \in \mathbb C^*$ such that $\delta_{ij}=\delta_i/\delta_j$ and, in this way, we can define, unless multiplication by constant factor, the {\it residual divisor}
$$
R=Res(\mathcal F, C)=\sum_{i=1}^n \delta_i C_i.
$$
In order to make precise the definition, we determine the residual divisor by the choice $\delta_1=1$, hence the solution is unique in $\delta_2, ..., \delta_n$ (obviously, this construction depends on the index order from the irreducible components of $C$).

\begin{example}
Suppose now
$$
D=\sum_{i=1}^n C_i+\sum_{k=1}^m S_k
$$
with $C_1,..., C_n$ compact smooth curves and $S_1,...,S_m$ the germs of separatrices which are not supported on $C$ trough reduced non degenerate singularities of $\mathcal F$ in $C-Sing(C)$. Suppose that the representation $\rho_{(\mathcal F, D)}$ is trivial, with {\it residual divisor}
$$
R=Res(\mathcal F, D)=\sum_{i=1}^n \mu_i C_i +\sum_{k=1}^m \delta_k S_k.
$$
Assume also that the foliation don't have weak separatrix supported on $C$.
Then
\begin{eqnarray*}
R\cdot C_j
&=&\sum_{i=1}^n \mu_i C_i\cdot C_j +\sum_{j=1}^m \delta_k S_k\cdot C_j
\nonumber\\
&=&\mu_j(C_j^2+\sum_{i\neq j} \frac{\mu_i}{\mu_j} C_i\cdot C_j +\sum_{k=1}^m \frac{\delta_k}{\mu_j} S_k\cdot C_j)
\nonumber\\
&=&\mu_j(C_j^2-CS(\mathcal F, C_j))
\nonumber\\
&=&0
\end{eqnarray*}
for $j=1,..., n$.
\end{example}

\section{Existence of Separatrix}\label{chapter: ES}

\subsection{Trivial Residual Representation and Separatrices}\label{S:Representação Residual Trivial e Separatrizes}

Let $\mathcal F$ be a holomorphic foliation on a smooth complex surface $X$. Suppose that $\mathcal F$ leaves invariant a {\it compact} curve $C=\sum_{i=1}^n C_i$ with the following properties:
\begin{enumerate}
\item $C_1,..., C_n$ are smooth curves;
\item  the singularities of $\mathcal F$ in $C$ are all reduced;
\item $Sing(C)=\bigcup_{i\neq j}C_i\cap C_j$ are reduced non degenerate singularities of $\mathcal F$.
\end{enumerate}

Write $Sep(\mathcal F, C)=S_1+...+S_m$ for the separatrices (germs) not supported on $C$ and passing trough reduced non degenerate singularities in $Sing(\mathcal F)\cap (C-Sing(C))$.

Let $D=C+Sep(\mathcal F,C)$ and $\rho =\rho_{(\mathcal F, D)}:\pi_1(\Gamma)\rightarrow \mathbb C^*$ be the residual representation of $\mathcal F$ along $D$. Obviously, $\rho_{(\mathcal F, D)}$ is trivial if, and only if, $\rho_{(\mathcal F, C)}$ is trivial.

\begin{thm}\label{T:Número Separatrizes}%[{\bf Theorem \ref{T:Teorema 2}}]\label{T:Número Separatrizes}
Suppose that $\mathcal F$ has no weak separatrix supported on $C$ and that the residual representation $\rho=\rho_{(\mathcal F, D)}$ is trivial, with residues $\mu_1,..., \mu_n$
along $C$ and residues $\delta_1,...,\delta_m$ along $Sep(\mathcal F, C)$, that is, the residual divisor is
$$
R=Res(\mathcal F, D)=\sum_{i=1}^n \mu_i C_i+\sum_{k=1}^m\delta_k S_k.
$$

If the matrix of intersection $[C_i\cdot C_j]_{ij}$ of $C$ is invertible, then
$$\#Sep(\mathcal F, C)=m\geq dim_{\mathbb Q}(\sum_{i=1}^n \mu_i \mathbb Q + \sum_{k=1}^m \delta_k \mathbb Q)\geq 1$$
\end{thm}
\begin{proof}
Consider the matrix of intersection $A=[C_i\cdot C_j\oplus C_i\cdot S_k]$ of $C$ with $D=C+Sep(\mathcal F, C)$. Take the $\mathbb C$-divisor in $X$ given by $R=\sum_{j=1}^n \mu_jC_j+\sum_{k=1}^m \delta_kS_k$. Hence $C_i\cdot R=R\cdot C_i=0$, for $i=1,... ,n$. Therefore
$$A(\mu_1,...,\mu_n,\delta_1,...,\delta_m)=(0,...,0)\in \mathbb C^{n}.$$
Since the matrix of intersection $[C_i\cdot C_j]$ is invertible, then the matrix $A$ has rank $n$, hence the equality above correspond to $n$ $\mathbb Q$-linearly independent relations (with integer coefficients) between $\mu_1,...,\mu_n,\delta_1,...,\delta_m$. Then
$$1\leq dim_{\mathbb Q}(\sum_{i=1}^n \mu_i \mathbb Q + \sum_{k=1}^m \delta_k \mathbb Q)\leq m=\#Sep(\mathcal F, C).$$
\end{proof}

An interesting consequence is the following

\begin{cor}\label{C:Existe Separatriz}
Suppose that the representation $\rho_{(\mathcal F, C)}$ is trivial. If the matrix of intersection $[C_i\cdot C_j]_{ij}$ of $C$ is invertible, then there is a separatrix $S$ trough $Sing(\mathcal F)\cap (C-Sing(C))$ not supported on $C$ and which is not a weak separatrix of saddle-node.
\end{cor}
\begin{proof}
In fact, if there is no saddle-node in $Sing(\mathcal F)\cap (C-Sing(C))$ with weak separatrix supported on $C$, then the Theorem \ref{T:Número Separatrizes} imply the existence of a singular point of the foliation in $C-Sing(C)$ trough which pass a separatrix not supported on $C$ and which is not a weak separatrix of saddle-node. Otherwise, if there is a saddle-node in $Sing(\mathcal F)\cap (C-Sing(C))$ with weak separatrix supported on $C$, then the  strong separatrix is not supported on $C$.
\end{proof}

\subsection{The Separatrix Theorem}\label{S:Separatrizes}

Let $\mathcal F$ be a singular holomorphic foliation on the smooth complex surface $X$. Suppose that $\mathcal F$ leaves invariant a curve $C=\sum_{i=1}^n C_i$ with the following properties (as before):
\begin{enumerate}
\item $C_1,..., C_n$ are smooth curves;
\item  the singularities of $\mathcal F$ in $C$ are all reduced;
\item $Sing(C)=\bigcup_{i\neq j} C_i \cap C_j$ are reduced non degenerate singularities of $\mathcal F$.
\end{enumerate}

A very interesting situation is when the representation $\rho=\rho_{(\mathcal F, C)}$ associated to the {\it exceptional} curve  is trivial. Remember this notion:

\begin{definition} [{\it Exceptional Curves}]
A compact, connected, reduced curve $C$ in a non singular surface $X$ is called {\it exceptional} if there is a bimeromorphism $\pi:X\rightarrow Y$ such that $C$ is exceptional for $\pi$, i.e., if there is an open neighbourhood $U$ of $C$ in $X$, a point $y\in Y$, and a neighbourhood $V$ of $y$ in $Y$, such that $\pi$ sends $U-C$ biholomorphically over $V-\{y\}$, where $\pi(C)=y$. We will also express this situation by saying that $C$ is {\it contracted} in $y$.
\end{definition}

The following classical result is very important:

\begin{thm}[{\bf Grauert's criterion}, \cite{BPV}, page 91]\label{Grauert Excepcional}
A compact, connected, reduced curve $C=\sum_{i=1}^n C_i$ with irreducible components $C_i$ in a smooth surface is exceptional if, and only if, the matrix of intersection $[C_i\cdot C_j]$ is negative definite.
\end{thm}
 
Now we suppose triviality of the residual representation and that the curve is exceptional. The theorem bellow is in fact a simple consequence of Corollary \ref{C:Existe Separatriz}, because every negative definite matrix is invertible. %, but we prefer to give an independent proof following similar ideas.

\begin{thm}\label{T:Trivial Geral}
If the curve $C$ is compact and exceptional and the representation $\rho=\rho_{(\mathcal F, C)}$ is trivial, then there is at last one singularity $p\in Sing(\mathcal F) \cap (C-Sing(C))$ and separatrix trough $p$ not supported on $C$ and which is not a weak separatrix of saddle-node.
\end{thm}

As a consequence we have a criterion for the existence of separatrix in singular (normal) surface.

\begin{cor}\label{C:Separatriz}
Let $\mathcal F$ be a foliation on the singular normal surface $X$. Let $\pi:Y\rightarrow X$ be a bimeromorphism which resolve the singularity $p$ and such that the induced foliation $\mathcal G=\pi^*\mathcal F$ is reduced.
Suppose that there exists a connected sub-curve $C=\sum_{i=1}^n C_i\subset E=\pi^{-1}(p)$ such that the foliation $\mathcal G$ is reduced non degenerated at the crossing points $Sing(C)=\bigcup_{i\neq j} C_i \cap C_j$ and 
\[q\in  Sing(E)\cap (C-Sing(C))\Rightarrow CS(\mathcal G, C, q)=0.
\]
%$\mathcal G$ has no weak separatrix of saddle-node supported on $C$.
If the representation $\rho=\rho_{({\mathcal G},C)}$ is trivial, then $\mathcal F$ has a separatrix trough $p$.
\end{cor}
\begin{proof}
The curve $C$ is exceptional, hence we can use Theorem \ref{T:Trivial Geral}. Thus, there exists a singularity $q\in C-Sing(C)$ and separatrix (of $\mathcal G$) trough $q$ but not supported on $E$ (by hypothesis). Such a separatrix will project by $\pi$ to a separatrix of $\mathcal F$ trough $p$.
\end{proof}

\begin{cor}[Generalization of the Separatrix Theorem by Camacho, \cite{C}]\label{C:Generalização do Teorema da Separatriz}
Consider a complex normal irreducible surface $X$. Suppose the resolution graph of $X$ at a point $p$ is a tree, i. e., a finite contractible $1$-dimensional complex. Then any germ of holomorphic foliation, singular at $p$, admits an invariant analytic curve through $p$.
%Let $\mathcal F$ be a foliation on a singular normal surface $X$. Let $E$ be an exceptional divisor obtained by resolution of the singularity $p\in X$ and reduction of singularities of the foliation. If the dual graph of $E$ is a tree, then there exists  separatrix trough $p$.
\end{cor}
\begin{proof}
Let $E$ be an exceptional divisor obtained by resolution of the singularity $p\in X$ and reduction of singularities of the foliation. The dual graph of $E$ is a tree, by hypothesis. It's enough to show the existence of a connected sub-curve $C\subset E$ in the conditions of the Corollary above, because such sub-curve will has also contractible dual graph, hence the representation $\rho_{(\mathcal G, C)}$ will be trivial.

We follow the argument of M. Toma in \cite{Toma}. Let $\Lambda\subset \Gamma$ be the sub-graph obtained by the union of all edges in $\Gamma$ that do not connect vertices by a crossing in saddle-node, that is, we remove from $\Gamma$ the edges which correspond to crossing in saddle-node. Note that when we remove such edge from the connected tree we divide the tree in two connected components, one is the component of the strong separatrix and the other is the component of the weak separatrix. Making this process and always fixing the component correspondent to the strong separatrix, at the end we obtain a connected component $\Gamma_0\subset \Lambda$ without vertices which are weak separatrices. Then $\Gamma_0$ is the dual graph of a sub-curve $C$ with the desired properties.
\end{proof}

In particular, we have the important

\begin{cor}[Separatrix Theorem of Camacho and Sad, \cite{CSS}]\label{C:Teorema da Separatriz}
Let $\mathcal F$ be a foliation on the smooth surface $X$ and $p\in Sing(\mathcal F)$. Then there exists a separatrix of $\mathcal F$ through $p$.
\end{cor}

\subsection{Torsion Residual Representation and Separatrices}\label{S:Representação Residual Torção e Separatrizes}
We begin with the familiar {\it covering trick} in its simplest version (non-branched).
\begin{prop} Let $X$ be a complex manifold and let $L$ be a line bundle over $X$ of order $k>0$ in ${\rm Pic}(X)$, that is, $k$ is the smaller positive integer such that $L^{\otimes k}$ is trivial. Then there is a regular {\it cyclic} covering of order $k$, $g:Y\rightarrow X$, such that $g^*L$ is trivial.
\end{prop}
\begin{proof} See \cite{BPV}, page 54.
% Let $\tau:L\rightarrow X$ be the projection of the bundle. Consider a trivialising section $s$ of $L^{\otimes k}$ and let $f:L\rightarrow L^{\otimes k}$ be the function given by the $k^{th}$ tensor power $f(v)=v^{\otimes k}=v\otimes ... \otimes v$ ($k$ times).

%Then $Y=f^{-1}(s(X))\subset L$ is a submanifold of $L$ and the projection $g=\tau\vert_{Y}:Y\rightarrow X$ is a regular  {\it cyclic} covering of order $k$ (for generalization of this construction, see \cite{BPV}, page 54). The automorphisms of covering are the restrictions to $Y$ of the automorphisms of bundle in $L$ correspondents to multiplication by $k$-roots of unity.

%By construction, the bundle $g^*L$ is trivial. To see this, note that just the inclusion of $Y$ in $L$ gives to us a global trivialising section of $g^*L$, since each element $v\in Y\subset L$ is non zero as an element of $L$ and, therefore, is non zero as an element of $g^*L$.

%$$\xymatrix{
%g^*L\ar@{^{(}->}[r]\ar@{->}[d]&\tau^*L\ar@{->}[r] \ar@{->}[d]&L\ar@{->}[r]^{f}\ar@{->}[d]^-{\tau}&L^{\otimes k}\ar@{->}[d]\\
%Y\ar@{^{(}->}[r]^-{inclusion}\ar@/_0.4cm/[rr]_{g}&L\ar@{->}[r]^-{\tau} &X\ar@{=}[r]&X
%}$$
\end{proof}

Remember also the following result from Complex Geometry:

\begin{thm}[Grauert, {\bf Satz 5}, \cite{Uber} page 340]\label{T:Grauert Excepcional Geral}
Let $A$ be a compact connected analytic subset of the analytic variety $X$. Then $A$ is an exceptional variety if and only if it has a strongly pseudoconvex neighbourhood $G$ in $X$ such that $A$ is the maximal compact analytic subset of $G$.
\end{thm}

As a consequence, in the particular case of our interest, we obtain the following fact: if $g:Y\rightarrow X$ is a finite covering of the complex surface $X$ and $C\subset X$ is an exceptional curve, then $g^{-1}(C)$ is an exceptional curve in $Y$.

Suppose that the foliation $\mathcal F$ leaves invariant a curve $C=\sum_{i=1}^n C_i$ such that:
\begin{enumerate}
\item $C_1,..., C_n$ are smooth curves; 
\item  all singularities of $\mathcal F$ in $C$ are reduced non degenerated.
\end{enumerate}
With this hypotheses, we obtain a version of Theorem \ref{T:Trivial Geral} in the case of torsion residual representation.

\begin{thm}\label{T:Trivial Geral Generalizado}
If the $\mathcal F$-invariante curve $C=\sum_{i=1}^{n} C_i\subset X$ is compact and exceptional and the representation $\rho=\rho_{(\mathcal F, C)}$ is torsion of order $k$, then there is at last one singularity $p\in Sing(\mathcal F) \cap (C-Sing(C))$ and separatrix trough $p$, not supported in $C$.
\end{thm}
\begin{proof}
Supose, by contradction, that there is no singularity of the foliation in $C-Sing(C)$. Let $F=F_{\rho}=N^*_{ \mathcal F}\otimes \mathcal O_X(C)\vert_C$ be the line bundle over $C$ associated to $\rho$ (see the subsection \ref{S:Grupo de Picard de Curvas}) and consider a $\mathbb C^*$-flat extension $L$ of $F$ to a neighbourhood $U$ of $C$ in such a way that $L^{\otimes k}$ is trivial. Then, by the covering trick mentioned at the beginning, we obtain a cyclic covering of order $k$ from a surface $Y$ over $X$, say $g:Y\rightarrow X$, with $g^*L$ trivial.

Consider $\mathcal G=g^*\mathcal F$ and $D=\sum_{j=1}^{kn} D_j=g^{-1}(C)$. Since we assume $Sing(\mathcal F)\cap(C-Sing(C))= \varnothing$, then also $Sing(\mathcal G)\cap(D-Sing(D))= \varnothing$. Thus
$$
N^*_{\mathcal G}\otimes \mathcal O_X(D)\vert_D=g^*F=g^*L\vert_D=\mathcal O_D
$$
hence the representation $\tilde{\rho}=\rho_{(\mathcal G, D)}$ is trivial (see the subsection \ref{S:Grupo de Picard de Curvas} and the Proposition \ref{P:Resíduo e Representação Generalizados}). By Theorem \ref{T:Grauert Excepcional Geral}, the curve $D$ is exceptional. Therefore, from Theorem \ref{T:Trivial Geral}, $Sing(\mathcal G)\cap(D-Sing(D))\neq \varnothing$, a contradction.
\end{proof}

\subsection{Foliations with $\mathbb Q$-Gorenstein Normal Bundle}\label{S:Folheações com Fibrado Normal Gorenstein}

A sheaf $\mathcal S$ on a singular surface is called {\it $\mathbb Q$-Gorenstein} if there is a positive integer $k>0$ such that the $k$-th tensor power $\mathcal S^{\otimes k}$ is locally trivial.% (even at the singularities).

Now we develop the proof of the {\bf Theorem \ref{T:Teorema 1}}:
\begin{thm}[{\bf Theorem \ref{T:Teorema 1}}]\label{T:Gorenstein}
Let $\mathcal F$ be a foliation on the normal singular surface $X$. If the foliation has no saddle-node in its resolution over the singularity $p\in X$ and the normal sheaf $N_{\mathcal F}$ is $\mathbb Q$-Gorenstein, then $\mathcal F$ has a separatrix through $p$.
%Let $\mathcal F$ be a foliation on the singular normal suface $X$. Let $f:Y\rightarrow X$ be a bimeromorphism which resolve the singularity $p\in X$ and such that the induced foliation $\mathcal G=\pi^*\mathcal F$ is reduced along $E=f ^{-1}(p)$.
%If $\mathcal G$ has no saddle-node in the exceptional divisor $E$ and the normal sheaf of $\mathcal F$ is $\mathbb Q$-Gorenstein, then $\mathcal F$ has a separatrix trough $p$.
\end{thm}

To prove the theorem, we need some previous results.

\begin{lemma}[{\bf Lemma 5} of \cite{Ueda}]\label{Extension}
Let $X$ be a  smooth compact complex surface, $K$ a compact subset of $X$ and $B$ be a holomorphic vector bundle over $X$. If $X-K$ is strictly pseudoconvex, then every section $s$ of $B$ over $X-K$ can be extended as a meromorphic section $\tilde{s}$ over all $X$.
\end{lemma}

\begin{prop}\label{P:Gorenstein}
Let $\mathcal F$ be a foliation on the singular normal surface $X$. Let $f:Y\rightarrow X$ be a resolution singularity at $p\in X$.
If the normal sheaf $N_{\mathcal F}$ is $\mathbb Q$-Gorenstein, then, in neighbourhood $V$ of $E$,
$$N_{\mathcal G}^{\otimes k}\vert_V=\mathcal O_Y(\sum_{i=1}^n a_iE_i)\vert_V$$
where $k$ is the smallest integer for which $N_{\mathcal F}^{\otimes k}$ is trivial in neighbourhood of $p$, the $a_i$ are integers and $f^{-1}(p)=E=\sum_{i=1}^n E_i$, where each $E_i$ is an irreducible component.
\end{prop}
\begin{proof}
First note that from Theorem \ref{T:Grauert Excepcional Geral} the exceptional divisor $E$ has a strictly pseudoconvex neighbourhood $V$. For a sufficiently small neighbourhood $U$ of $p$, the sheaf $N_{\mathcal F}^{\otimes k}$ is trivial. For $U$ and $V$ small enough, we can assume $f^{-1}(U)=V$.

Since $N_{\mathcal G}^{\otimes k}=f^*N_{\mathcal F}^{\otimes k}$ is trivial over $V-E$, then there is a holomorphic section without zeros $s$ of the bundle $N_{\mathcal G}^{\otimes k}$ over $V-E$. The Lemma \ref{Extension} above imply that $s$ can be extended as a meromorphic section $\tilde s$ over all $V$, with polos or zeros along $E$. This finish the proof.
\end{proof}

%\begin{lemma}\label{L:Gorenstein Fraco}
%Let $\mathcal G$ be a reduced foliation on $Y$ tangent to the curve $E=\sum_{i=1}^m E_i$, where each $E_i$ is a smooth curve in $Y$.
%Suppose that $\mathcal G$ has no sadlle-node in $E$ and that $Sing(\mathcal G)\cap E=Sing(E)=\bigcup_{i\neq j} E_i \cap E_j.$ If $N^*_{\mathcal G}\otimes \mathcal O_Y(E)\vert_{E}^{\otimes k}=\mathcal O_E$ for some integer $k>0$, then the matrix of intersection of $E$ is not negative definite ($E$ is not exceptional).
%\end{lemma}
%\begin{proof}
%If $N^*_{\mathcal G}\otimes \mathcal O_Y(E)\vert_{E}^{\otimes k}=\mathcal O_E$, then, by Proposition \ref{P:Resíduo e Representação Generalizados}, the representation $\rho_{(\mathcal G,E)}$ is torsion. Then by Theorem \ref{T:Trivial Geral Generalizado} the curve $E$ cannot be exceptional.
%\end{proof}

\begin{proof}[{\bf Proof of the Theorem \ref{T:Gorenstein}}]
Let $f:Y\rightarrow X$ be a resolution singularity at $p\in X$ and such that the induced foliation $\mathcal G=\pi^*\mathcal F$ is reduced along $E=f ^{-1}(p)$. Suppose, by contradiction, that there is no singular point of $\mathcal G$ in $E-Sing(E)$. Since $N_{\mathcal F}$ is a $\mathbb Q$-Gorenstein sheaf, it follows from Proposition \ref{P:Gorenstein} that, in neighbourhood of $E$,
$$N_{\mathcal G}^{\otimes k}=\mathcal O_Y(\sum_{i=1}^n a_iE_i)$$
where $k>0$, and the $a_i$ are integers, and $E=\sum_{i=1}^n E_i$, each $E_i$ an irreducible component of $E$.

Since there is no saddle-node in $E$, then, by the intersection formula of the normal bundle given by Theorem \ref{T:Normal Intersecta}, we have $N_{\mathcal G}\cdot E_j=E\cdot E_j$, that is, $(N_{\mathcal G}-E)\cdot E_j=0$ for $j=1,...,n$. Then $(\sum_{i=1}^n a_iE_i-kE)\cdot E_j=0$, that is, $(\sum_{i=1}^n (a_i-k)E_i)\cdot E_j=0$ for $j=1,...,n$. Hence
$$(\sum_{i=1}^n (a_i-k)E_i)^2=0.$$

Therefore, since the matrix of intersection of $E$ is negative definite (Grauert's Criterion \ref{Grauert Excepcional}), $a_i=k$ for $i=1,...,n$, that is,
$$N_{\mathcal G}^{\otimes k}=\mathcal O_Y(k\sum_{i=1}^n E_i)=\mathcal O_Y(kE)$$ 
Hence by Theorem \ref{T:Trivial Geral Generalizado} there is a singular point of $\mathcal G$ in $E-Sing(E)$ and  separatrix trough it not supported on $E$, and which is projected over a separatrix of $\mathcal F$ trough $p$.
\end{proof}

The proof just presented shows something more:

\begin{thm}\label{T:Quase Gorenstein}
Let $\mathcal F$ be a foliation on the singular normal surface $X$. Let $f:Y\rightarrow X$ be a resolution singularity at $p\in X$ and such that the induced foliation $\mathcal G=f^*\mathcal F$ is reduced along $E=f^{-1}(p)$. If $\mathcal G$ has no saddle-node in the exceptional divisor $E$ and
$$N_{\mathcal G}\vert_E^{\otimes k}=\mathcal O_Y(\sum_{i=1}^n a_iE_i)\vert_E$$
where $a_1,...,a_n\in \mathbb Z$, then $\mathcal F$ has a separatrix trough $p$.
\end{thm}

%Uma pergunta natural que surge é se poderíamos supor que o feixe tangente $T_{\mathcal F}$ é $\mathbb Q$-Gorenstein e obter, nas hipóteses do Teorema \ref{T:Teorema 1}, existência de separatriz. Conforme \cite{Adolfo}, exemplos naturais respondendo negativamente à questão surgem de campos vetoriais em {\it superfícies de Kato intermediárias}. Superfícies de Kato intermediárias possuem um divisor $D$ (o divisor reduzido maximal de curvas racionais) com suporte conexo e matriz de interseção negativa definida. Algumas delas admitem um campo de vetores holomorfo global $v$. Como a matriz de interseção é negativa definida, então $D$ é invariante pelo campo $v$. Tais campos vetorias são bem entendidos e sabemos que em vizinhança de $D$ as únicas singularidades de $v$ estão em $D$ e qualquer germe de curva invariante passando por $D$ está suportado em $D$. Assim, contraindo $D$ num ponto $p$, obtemos uma superfície singular (em $p$) admitindo um campo de vetores holomorfo sem separatriz passando por $p$. Ou seja, uma folheação $\mathcal F$ numa superfície singular sem separatriz passando por $p$, mas cujo feixe tangente $T_{\mathcal F}$ é Gorenstein.

\bibliographystyle{amsplain}

\end{document}